\documentclass[12pt,oneside,a4paper]{amsart}

\usepackage[T1]{fontenc}
\usepackage{ae}
\usepackage[english]{babel}
\usepackage{amsmath}
\usepackage{mathtools}
\usepackage{amssymb}
\usepackage{amsthm}
\usepackage{mathrsfs}
\usepackage{enumitem}
\usepackage[foot]{amsaddr}

\usepackage[utf8]{inputenc}  
\usepackage{amscd}
\usepackage{mathrsfs}
\usepackage[all]{xy}
\usepackage[OT1]{fontenc}
\usepackage{amssymb,latexsym}
\usepackage{type1cm,ae}
\usepackage{bbding}
\usepackage[dvipsnames]{xcolor}
\usepackage{pgf,tikz}
\usepackage{changepage}
\usetikzlibrary[arrows]
\usetikzlibrary{decorations.pathreplacing}

\usepackage{afterpage}
\usepackage[doi=false,isbn=false,url=false,backend=biber,style=alphabetic,maxbibnames=99,minalphanames=5,maxalphanames=10,giveninits=true]{biblatex}

\DeclareFieldFormat{pages}{#1}
\renewbibmacro{in:}{%
	\ifentrytype{article}
	{}
	{\bibstring{in}%
		\printunit{\intitlepunct}}}
\AtEveryBibitem{
	\ifentrytype{book}{%
		\clearfield{pages}%
	}{%
	}%
}

\usepackage{graphicx}
\usepackage{xspace}
\usepackage{setspace}
\usepackage[english]{babel}
\usepackage{tikz}
\usepackage{etoolbox}
\usepackage[hidelinks,linktocpage,colorlinks=true,linkcolor=blue,citecolor=red]{hyperref}

\newcommand{\half}{\frac{1}{2}}
\newcommand{\e}{\varepsilon}
\newcommand{\p}{\partial}

\newcommand{\va}{\enspace}

\newcommand{\Om}{\Omega}
\newcommand{\N}{\mathbb{N}}
\newcommand{\Rnn}{\mathbb{R}^{n}}
\newcommand{\R}{\mathbb{R}}

\newcommand{\chr}[2]{\Gamma^{#1}_{\hphantom{#1}#2}}
\newcommand{\chrhat}[2]{\hat{\Gamma}^{#1}_{\hphantom{#1}#2}}

\DeclareMathOperator{\spt}{spt}

\DeclareMathOperator{\dive}{div}

\DeclareMathOperator{\Tr}{Tr}
\DeclarePairedDelimiterX{\norm}[1]{\lVert}{\rVert}{#1}

\makeatletter
\renewenvironment{proof}[1][\proofname]{%
	\par\pushQED{\qed}\normalfont%
	\topsep6\p@\@plus6\p@\relax
	\trivlist\item[\hskip\labelsep\bfseries#1\@addpunct{.}]%
	\ignorespaces
}{%
	\popQED\endtrivlist\@endpefalse
}
\makeatother

\newtheorem{thm}{Theorem}[section]

\newtheorem{lemma}[thm]{Lemma}

\newtheorem{defin}[thm]{Definition}

\theoremstyle{definition}

\numberwithin{equation}{section}

\author{Janne Nurminen}
\address{Department of Mathematics and Statistics, University of Jyväskylä}
\email{janne.s.nurminen@jyu.fi}

\begin{document}
	
	\title[AN INVERSE PROBLEM FOR THE MINIMAL SURFACE EQUATION]{AN INVERSE PROBLEM FOR THE MINIMAL SURFACE EQUATION IN THE PRESENCE OF A RIEMANNIAN METRIC}

	\begin{abstract}
		In this work we study an inverse problem for the minimal surface equation on a Riemannian manifold $(\R^{n},g)$ where the metric is of the form $g(x)=c(x)(\hat{g}\oplus e)$. Here $\hat{g}$ is a simple Riemannian metric on $\R^{n-1}$, $e$ is the Euclidean metric on $\R$ and $c$ a smooth positive function. We show that if we know the associated Dirichlet-to-Neumann maps corresponding to metrics $g$ and $\tilde{c}g$, then the Taylor series of the conformal factor $\tilde{c}$ at $x_n=0$ is equal to a positive constant. We also show a partial data result when $n=3$.
		
		\vspace{5pt}
	\end{abstract}
	\maketitle
	
	\tableofcontents

	
	\section{Introduction}
	
	This article focuses on an inverse problem for the minimal surface equation (MSE), which is a quasilinear elliptic PDE. In particular we consider MSE on a manifold $(\Rnn,g), n\geq3$, where
	\begin{equation}\label{metricg_intro}
		g(x',x_n)=c(x',x_n)\begin{pmatrix}
			\hat{g}(x') & 0 \\
			0 & 1
		\end{pmatrix}.
	\end{equation}
	Here $(x',x_n)\in\R^{n-1}\times\R$, $\hat{g}$ is a Riemannian metric on $\R^{n-1}$ and $c\in C^{\infty}(\Rnn),$ $c(x)>0$ for all $x\in\Rnn$. Inverse problems for the MSE were considered in \cite{Nurminen2022} where $\hat{g}$ was Euclidean and in \cite{Carstea2022} where $c=1$. Now we consider metrics of the form \eqref{metricg_intro} which includes both of these cases. Metrics with local coordinates of the form \eqref{metricg_intro} were introduced for example in \cite{Ferreira2009b}, \cite{Ferreira2016} where the Calderón problem was studied.
	
	Let $\Om\subset\R^{n-1}$ be a bounded domain. In order to define a minimal surface, we introduce notation for Riemannian Hessian and norm of the gradient
	\begin{align*}
		\nabla_g^2u=\left(\partial^2_{x_ix_j}u - \chr{m}{ij}\partial_{x_m}u\right)_{i,j=1}^n,\quad |\nabla_gu|_g^2=g^{ij}\partial_{x_i}u\partial_{x_j}u.
	\end{align*}
	We use the Einstein summation convention where there is no chance to misinterpret. Also here $\chr{m}{ij}$ are the standard Christoffel symbols.
	Now define the Laplace-Beltrami operator to be the trace of the Hessian
	\begin{equation*}
		\Delta_gu=\Tr(\nabla_g^2u)=g^{ij}\left(\partial^2_{x_ix_j}u-\chr{m}{ij}\partial_{x_m}u\right).
	\end{equation*}
	Then for the metric \eqref{metricg_intro} the MSE has the form
	\begin{align}\label{mse_intro}
		&\dive_g\left(\frac{\nabla_{g}u}{\left(1+|\nabla_{g}u|_g^2\right)^{\half}}\right) \\\notag
	&+ \frac{\Delta_{g}u(1-c^{-1}) + \frac{1}{2c}(\nabla_{g}u)^j\p_{x_j}c(1 + |\nabla_{g}u|_g^2 - c^{-2}) + \frac{n-1}{2c^3}\p_{x_n}c(1 + |\nabla_{\hat{g}}u|^2_{\hat{g}})}{\left(1+|\nabla_{g}u|_g^2\right)^{\frac{3}{2}}} = 0.
	\end{align}
	For the derivation of this equation, see Section \ref{section derivation}. When $c=1$, one gets the MSE from \cite{Carstea2022}. The graph of a function $u\colon\Om\subset\R^{n-1}\to\R$ is called a \textit{minimal surface} if it satisfies \eqref{mse_intro} for all $x'\in\Om$.
	
	We will look at the boundary value problem
	\begin{equation}\label{BVP_main_intro}
		\left\{\begin{array}{ll}
			F(x',u,\nabla_g u,\nabla_g^2u)=0 & \text{in}\,\, \Om \\
			u=f, & \text{on}\,\, \partial\Om
		\end{array} \right.
	\end{equation}
	where $F\colon \R^{n^2} \to \R,$
	\begin{align*}\label{def_F}
		F(x',u,\nabla_g u,\nabla_{g}^2u):=&  \left(- \Delta_{\hat{g}}u + \frac{1-n}{2c}\hat{g}^{mr}\p_{x_r}c\p_{x_m}u + \frac{n-1}{2c}\p_{x_n}c\right)(1 + |\nabla_{\hat{g}}u|^2_{\hat{g}}) \\
		&+ \nabla_{\hat{g}}^2u(\nabla_{\hat{g}}u,\nabla_{\hat{g}}u)
	\end{align*}
	(this is an equivalent formulation of \eqref{mse_intro}, see Section \ref{section derivation}).
	One can show that \eqref{BVP_main_intro} is well-posed for small enough boundary data $f$ for example by following the arguments in \cite{MR4188325}. In particular one can show that there exist $C, \delta>0$ and $s>3$, $s\notin\N$, such that for all 
	\[
	f\in U_{\delta}:=\{h\in C^s(\partial\Om):||h||_{C^s(\partial\Om)}<\delta\}
	\]
	there is a unique small solution $u_f$ in $\{v\in C^s(\bar{\Om}):||v||_{C^s(\bar{\Om})}<C\delta\}$. In addition to full data we will be considering partial data and for this let $\Gamma\subset\p\Om$ be a nonempty open subset of the boundary. We can now define a (partial) Dirichlet-to-Neumann map (DN map) $\Lambda_g^{\Gamma}\colon C_c^s(\Gamma)\to C^{s-1}(\Gamma),$
	\begin{equation}\label{DN_map_intro}
		\Lambda_g^{\Gamma}f\mapsto \p_{\nu}u_f|_{\Gamma}
	\end{equation}
	where $u_f$ is the unique small solution corresponding to the boundary value $f$, $\p_{\nu}u=\hat{g}^{ij}\p_{x_i}u\nu_j$ and $\nu$ is the unit normal of $\p\Om$. When $\Gamma=\p\Om$, we will denote $\Lambda_g=\Lambda_g^{\p\Om}$.
	
	The inverse problem we study in this work is the following: Given the knowledge of the partial DN map for two metrics in the same conformal class, does it hold that the metrics are the same? We note that when trying to recover the metric from boundary measurements one expects an obstruction to uniqueness by a boundary fixing diffeomorphism. When the metrics are in the same conformal class such an obstruction is not present, since the only diffeomorphism that leaves a conformal class invariant is the identity. In this article we will give some partial answers to this inverse problem.
	
	Before stating the results, we will introduce the class of \textit{admissible metrics} \cite{Ferreira2009b}.
	\begin{defin}
		A Riemannian metric $g$ is called admissible if on $\Rnn$ it is of the form \eqref{metricg_intro} and $(\Om,\hat{g})$ is a simple manifold.
	\end{defin}


	\begin{defin}
		A compact manifold $(\hat{M},\hat{g})$ with boundary is called simple if for any $p\in\hat{M}$ the exponential map $\exp_p$ with its maximal domain of definition in $T_p\hat{M}$ is a diffeomorphism onto $\hat{M}$, and if $\p\hat{M}$ is strictly convex in the sense that its second fundamental form is positive definite.
	\end{defin}
	
	The first result is for the full data case.
	
	\begin{thm}\label{main thm}
		Let $(\Rnn,g)$, $n\geq3$, be a Riemannian manifold where $g$ is as in \eqref{metricg_intro} with $\p_{x_n}c(x',0) = \p_{x_n}^2c(x',0) = 0$ and let $\Om\subset\R^{n-1}$ be a bounded domain with $C^\infty$ boundary. When $n=3$ assume that $\Om$ is simply connected and when $n>3$ assume that $\hat{g}$ is admissible. Let $\tilde{c}\in C^{\infty}(\Rnn)$ be such that $\p_{x_n}\tilde{c}(x',0) = \p_{x_n}^2\tilde{c}(x',0) = 0$ and let $\p_{x_n}^k\tilde{c}(x',0)|_{\p\Om}=0$ for $k\geq3$. Assume that $\Lambda_g(f) = \Lambda_{\tilde{c}g}(f)$ for all $f\in U_{\delta}$, where $\delta>0$ is sufficiently small. Then
		\begin{equation*}
			\tilde{c}(x',0) = \lambda, \quad \p_{x_n}^k\tilde{c}(x',0) = 0
		\end{equation*}
		for some $\lambda>0$ and for all $k>2$.
		
		Thus, if $\tilde{c}$ is real analytic with respect to $x_n$, then $\tilde{c}(x)=\lambda$ for all $x\in\Om\times\R$.
	\end{thm}
	
	For the partial data we will only consider the case $n=3$. For $n>3$ we would be able to get some results, but the partial data results for the magnetic Schrödinger equation (see \cite{Selim2023}) are more restrictive and hence we do not record these results here.
	
	\begin{thm}\label{main thm_partial}
		Let $(\R^3,g)$, $\tilde{c}$ and $\Om$ be as in Theorem \ref{main thm} and let $\Gamma\subset\p\Om$, $\Gamma\neq\emptyset$, be open. Assume that $\Lambda_g^{\Gamma}(f) = \Lambda_{\tilde{c}g}^{\Gamma}(f)$ for all $f\in U_{\delta}$, where $\delta>0$ is sufficiently small. Then
		\begin{equation*}
			\tilde{c}(x',0) = \lambda, \quad \p_{x_3}^k\tilde{c}(x',0) = 0
		\end{equation*}
		for some $\lambda>0$ and for all $k>2$.
		
		Thus, if $\tilde{c}$ is real analytic with respect to $x_3$, then $\tilde{c}(x)=\lambda$ for all $x\in\Om\times\R$.
	\end{thm}

	These results extend the results from \cite{Nurminen2022}.
	
	We make some comments about the assumptions made in Theorems \ref{main thm} and \ref{main thm_partial}. First of all, the assumption $\p_{x_n}c(x',0) = \p_{x_n}\tilde{c}(x',0) = 0$ is needed for the well-posedness of the boundary value problem \eqref{BVP_main_intro} for small data, that is we want it to be well-posed for both the metric $g$ and the metric $\tilde{c}g$. The assumption $\p_{x_n}^2c(x',0) = \p_{x_n}^2\tilde{c}(x',0) = 0$ is needed in order for the method used in the proof to work and it is not yet known if this could be removed. Lastly, there is a small gauge invariance in \eqref{mse_intro} and also in the DN map. To be more precise, if one replaces the conformal factor $c$ in \eqref{metricg_intro} by $\mu c$, $\mu\neq0$, then both the equation \eqref{mse_intro} and the DN map remain the same. Thus we cannot have $\lambda=1$ in Theorem \ref{main thm}.
	
	There are also some assumptions made when $n=3$. The assumption that $\Om$ is simply connected is needed to use the Poincaré Lemma. This is something that one might be able relax with some amount of work by looking at the so called gauge isomorphism mentioned in \cite{Guillarmou2011} that relates two connection 1-forms. The other assumption that $\p_{x_n}^k\tilde{c}(x',0)|_{\p\Om}=0$ for $k\geq3$ could possibly be removed by doing a boundary determination result.
	
	The proof will use the method of higher order linearization. We will begin by looking at the first linearization of \eqref{BVP_main_intro} and the DN map for that. The first linearization of the DN map corresponds to an advection diffusion type equation. From this it is possible to determine the advection term (\cite{Guillarmou2011} for $n=3$, \cite{Krupchyk2018} for $n>3$) and this will then imply that $\tilde{c}(x',0) = \lambda$ for some $\lambda\in(0,\infty)$. Now fix $x_0\in\Om$ and construct a solution $v^{(0)}$ to the adjoint of the linearized equation so that $v^{(0)}(x_0)\neq0$. Then using the higher order linearizations one can obtain an integral identity
	\begin{equation*}
		\int_{\Om}\tilde{c}(x',0)^{-1}\partial_{x_n}^{N+2}\tilde{c}(x',0)v^{(0)}\prod_{k=1}^{N+1}v^{l_k}\,dx'=0,
	\end{equation*}
	where $N+1$ is the order of linearization and $v^{l_k}$ are solutions to the linearized equation. Then choosing $v^{l_k}\equiv1$ for $k>2$ and using that a product of these solutions form a complete set in $L^1$ (see Lemma \ref{lemma_density}) we can conclude that
	\begin{equation*}
		\tilde{c}(x',0)^{-1}\partial_{x_n}^{N+2}\tilde{c}(x',0)v^{(0)}=0.
	\end{equation*}
	From this, using $v^{(0)}(x_0)\neq0$ and that $x_0$ was arbitrary, we get $\p_{x_n}^{N+2}\tilde{c}(x',0) = 0$. Then one proceeds by induction.
	
	For the partial data, the proof is very similar and has only a few modifications. One of the more significant changes is that instead of using the results of \cite{Guillarmou2011} to determine the advection term, we use the results for partial data in \cite{Tzou2017}.
	
	As mentioned above, the main technique in this work is the method of higher order linearization. This method, which uses the nonlinearity of the partial differential equation as a tool, was first introduced in \cite{MR3802298}  in the case of a nonlinear wave equation and was further developed in \cite{MR4188325}, \cite{MR4104456} for nonlinear elliptic equations. In these works the equation $\Delta u + a(x,u)=0$ was considered and further results for example in partial data and obstacle problems were obtained in \cite{MR4052205}, \cite{MR4269409}. Also there has been multiple works when the nonlinearity is of power type (e.g. \cite{MR4332042}, \cite{Salo2022} and \cite{Nurminen2023}).
	
	At the same time there have been multiple articles concerning inverse problems for other nonlinear elliptic PDEs. For example different cases of nonlinear conductivity equations have been considered in \cite{MR4300916}, \cite{kian2020partial} and in the latter they also consider partial data. For the nonlinear magnetic Schrödinger equation the full data case has been treated on conformally transversally anisotropic manifolds in \cite{Krupchyk2020} and in the Euclidean setting the partial data case is treated in \cite{lai2020partial}.
	
	Linearization has been used before these works mentioned above. It was used in the parabolic case in \cite{MR1233645} where it was shown that the first linearization of the DN map is the DN map for a linear equation. Other nonlinear elliptic cases have been treated, for example in \cite{MR1295934}, \cite{MR1465069}.
	
	The present work focuses on an inverse problem for the minimal surface equation and for this equation there have been some works previously in the literature. The Euclidean case has been treated in \cite{Munoz2020}, in the sense that they deal with a quasilinear conductivity depending on a function $u$ and its gradient $\nabla u$. After that there have been two works on the minimal surface equation in a Riemannian manifold setting, one by the author of this article \cite{Nurminen2022} and another one in \cite{Carstea2022}. These were simultaneously and independently done. The work of this article somehow brings these two articles together in the following sense: In \cite{Nurminen2022} the metric was conformally Euclidean and thus in this article the setting is more general. In \cite{Carstea2022} the authors consider a two dimensional Riemannian manifold $(\Sigma,g)$ and look at a minimal surface $Y\subset \Sigma\times\R$ given as a graph of a function. Thus if $\Sigma\subset\R^2$ and in this article we would have $c\equiv1$, the settings in these two articles would be the same.
	
	There is also the work \cite{Alexakis2020} that is related to minimal surfaces and inverse problems. In that article the authors consider a generalization of the boundary rigidity problem in the sense that instead of measuring distances of boundary points they use measurements related to areas of minimal surfaces.
	
	The rest of this article is organized as follows. In Section \ref{section derivation} we derive the minimal surface equation in the setting mentioned above. Section \ref{section linearizations} is dedicated to calculating the first and second order linearizations of \eqref{BVP_main_intro} and addressing shortly the well-posedness of \eqref{BVP_main_intro}. Finally in Section \ref{section proof of main} we prove the main results and two Lemmas to help in the proof.
	\\
	
	\textbf{Acknowledgements.} The author was supported by the Finnish Centre of Excellence in Inverse Modelling and Imaging (Academy of Finland grant 284715).
	The author would like to thank Mikko Salo for helpful discussions on the minimal surface equation and everything related to inverse problems.

	\section{Deriving the minimal surface equation}\label{section derivation}
	
	In this section we derive the equation \eqref{mse_intro}. This is done similarly as in \cite[Section 3]{Nurminen2022}. Let $(M,g),$ $M=\Rnn$, $n\geq3$, be a Riemannian manifold with the metric
	\begin{equation}\label{metricg}
		g(x',x_n)=c(x',x_n)\begin{pmatrix}
			\hat{g}(x') & 0 \\
			0 & 1
		\end{pmatrix}
	\end{equation}
	where
	\begin{equation*}\label{metricc}
		(x',x_n)\in\R^{n-1}\times\R,\; c\in C^{\infty}(\R^n),\; c(x)>0\quad\text{for all}\quad x\in\R^n
	\end{equation*}
	and $\hat{g}$ is a Riemannian metric on $\R^{n-1}$. 
	We use the standard notations $g_{ij}$ for the matrix $g$ and $g^{ij}$ for the inverse $g^{-1}$.
	These assumptions are valid for the rest of the article, unless otherwise stated.
	
	Let $u\colon \Om\subset\R^{n-1}\to\R,$ $u\in C^2(\bar{\Om})$, and consider the graph of the function $u$
	\begin{equation*}
		\text{Graph}_u=\{(x',u(x'))\colon x'\in\Om\}\subset M.
	\end{equation*}
	This graph is a minimal surface if and only if its mean curvature $H$ is equal to zero at all points on the graph. By defining 
	\begin{equation*}
		f\colon\Om\times\R\to\R,\quad f(x',x_n)=x_n-u(x'),
	\end{equation*}
	the graph of $u$ is the surface 
	\begin{equation*}
		\Sigma:=\{(x',x_n)\in\Om\times\R : f(x',x_n)=0\}.
	\end{equation*}
	The mean curvature of $\Sigma$ at $x\in\Sigma$ is the sum of principal curvatures. We omit the normalizing factor $\frac{1}{n-1}$ when calculating the mean curvature and we use the Einstein summation convention when it does not cause confusion. In order to calculate the principal curvatures, we introduce the Riemannian gradient and Hessian of a function $f\colon M\to\R$:
	\begin{align*}
		\nabla_gf=g^{ij}\partial_{x_i}f\partial_{x_j},\quad \nabla_g^2f=\left(\partial^2_{x_ix_j}f-\chr{m}{ij}\partial_{x_m}f\right)_{i,j=1}^n,
	\end{align*}
	where $g^{ij}$ is the inverse of $g_{ij}$ and $\chr{m}{ij}=\frac{1}{2}g^{ml}(\partial_{x_i}g_{jl}+\partial_{x_j}g_{il}-\partial_{x_l}g_{ij})$ is the Christoffel symbol related to the metric $g$.
	Define also the Laplace-Beltrami operator, which is the trace of the Hessian (this is one way of defining it), and the norm of the gradient:
	\begin{equation*}
		\Delta_gf=\Tr(\nabla_g^2f)=g^{ij}\left(\partial^2_{x_ix_j}f-\chr{m}{ij}\partial_{x_m}f\right),\quad |\nabla_gf|_g^2=g^{ij}\partial_{x_i}f\partial_{x_j}f .
	\end{equation*}
	Now the principal curvatures of $\Sigma$ at $x\in\Sigma$ are the eigenvalues of $\nabla_g^2f(x)$ restricted to the tangent space $T_x\Sigma$ at $x$. Since $\frac{\nabla_g f(x)}{|\nabla_g f(x)|_g}$ is a normal to $\Sigma$ at the point $x$, we have $T_x\Sigma=\{\nabla_g f(x)\}^{\perp}$, or in other words, the tangent space $T_x\Sigma$ is the orthogonal complement of the vector $\nabla_g f(x)$. 
	
	Let $\{E_1,\ldots,E_{n-1}\}$ be an $g$-orthonormal basis of $T_x\Sigma$. Then $\left\{E_1,\ldots,E_{n-1},\frac{\nabla_g f(x)}{|\nabla_g f(x)|_g}\right\}$ is an orthonormal basis of $\Rnn$. Now the mean curvature of $\Sigma$ at $x\in\Sigma$ is the trace of $\nabla_g^2f(x)|_{\{\nabla_g f(x)\}^{\perp}}$:
	\begin{align*}
		H(x)&=\sum_{i=1}^{n-1}\langle\nabla_g^2f(x)E_i,E_i\rangle\\
		&=\sum_{i=1}^{n-1}\langle\nabla_g^2f(x)E_i,E_i\rangle+\left(\nabla_g^2f(x)\right)\left(\frac{\nabla_g f(x)}{|\nabla_g f(x)|_g},\frac{\nabla_g f(x)}{|\nabla_g f(x)|_g}\right)\\
		&-\left(\nabla_g^2f(x)\right)\left(\frac{\nabla_g f(x)}{|\nabla_g f(x)|_g},\frac{\nabla_g f(x)}{|\nabla_g f(x)|_g}\right)\\
		&=\Tr(\nabla_g^2f(x))-|\nabla_gf(x)|_g^{-2}\left(\nabla_g^2f(x)\right)\left(\nabla_g f(x),\nabla_g f(x)\right)\\
		&=\Delta_gf(x)-|\nabla_gf(x)|_g^{-2}\left(\nabla_g^2f(x)\right)\left(\nabla_g f(x),\nabla_g f(x)\right).
	\end{align*}
	Thus $\text{Graph}_u$ is a minimal surface if and only if 
	\begin{equation}\label{general minsurfeq}
		|\nabla_gf(x)|_g^2\Delta_gf(x) - \left(\nabla_g^2f(x)\right)\left(\nabla_g f(x),\nabla_g f(x)\right)=0\quad\text{for all}\va x\in\text{Graph}_u.
	\end{equation}
	
	Next we will calculate the minimal surface equation more explicitly using the metric \eqref{metricg}. Now 
	\begin{equation*}
		g^{-1}=c^{-1}\begin{pmatrix}
			\hat{g}^{-1} & 0 \\
			0 & 1
		\end{pmatrix}
	\end{equation*}
	is the inverse matrix of \eqref{metricg}. Let us calculate the first term of \eqref{general minsurfeq}:
	\begin{align}\label{gen_minsurf_p1}
		&c^2|\nabla_gf(x)|_g^2\Delta_gf(x)\\\notag
		&= c^2g^{ij}(\delta_{in}-\p_{x_i}u)(\delta_{jn}-\p_{x_j}u)g^{kl}\left(-\p^2_{x_kx_l}u - \chr{m}{kl}(\delta_{mn}-\p_{x_m}u)\right)\\\notag
		&= c^2g^{ij}(\delta_{in}\delta_{jn} - \delta_{in}\p_{x_j}u - \delta_{jn}\p_{x_i}u + \p_{x_i}u\p_{x_j}u)g^{kl}(-\p^2_{x_kx_l}u - \chr{n}{kl} + \chr{m}{kl}\p_{x_m}u)\\\notag
		&=c\left(1+|\nabla_{\hat{g}}u|_{\hat{g}}^2\right)\left(- \Delta_{g}u - g^{kl}\chr{n}{kl}\right).
	\end{align}
	The second term of \eqref{general minsurfeq} becomes
	\begin{align}\label{gen_minsurf_p2}
		&c^2\left(\nabla_g^2f(x)\right)\left(\nabla_g f(x),\nabla_g f(x)\right)\\\notag
		&= c^2(\p_{x_ix_j}^2f - \chr{m}{ij}\p_{x_m}f)g^{ai}\p_{x_a}fg^{bj}\p_{x_b}f\\\notag
		&= c^2(-\p_{x_ix_j}^2u - \chr{n}{ij} + \chr{m}{ij}\p_{x_m}u)g^{ai}g^{bj}(\delta_{an}\delta_{bn} - \delta_{an}\p_{x_b}u - \delta_{bn}\p_{x_a}u + \p_{x_a}u\p_{x_b}u)\\\notag
		&= c^2(-\p_{x_ix_j}^2u - \chr{n}{ij} + \chr{m}{ij}\p_{x_m}u)(g^{ni}g^{nj} - g^{ni}g^{bj}\p_{x_b}u - g^{ai}g^{nj}\p_{x_a}u + g^{ai}g^{bj}\p_{x_a}u\p_{x_b}u)\\\notag
		&= \chr{m}{nn}\p_{x_m}u - \chr{n}{nn} - 2(\chr{m}{nj}\p_{x_m}u - \chr{n}{nj})cg^{bj}\p_{x_b}u\\\notag
		&+ c^2(-\p_{x_ix_j}^2u - \chr{n}{ij} + \chr{m}{ij}\p_{x_m}u)g^{ai}g^{bj}\p_{x_a}u\p_{x_b}u\\\notag
		&= \chr{m}{nn}\p_{x_m}u - \chr{n}{nn} - 2(\chr{m}{nj}\p_{x_m}u - \chr{n}{nj})\left(\nabla_{\hat{g}}u\right)^j\\\notag
		&- \nabla_{g}^2u(\nabla_{\hat{g}}u,\nabla_{\hat{g}}u) - \chr{n}{ij}\left(\nabla_{\hat{g}}u\right)^i\left(\nabla_{\hat{g}}u\right)^j.
	\end{align}
	In order to simplify further we will calculate the Christoffel symbols more explicitly:
	\begin{align*}
		2\chr{m}{ij}&= g^{mr}(\p_{x_j}g_{ir} + \p_{x_i}g_{jr}-\p_{x_r}g_{ij}),\\
		2\chr{m}{nn}&= g^{mr}(2\p_{x_n}g_{nr} - \p_{x_r}g_{nn}) = 2g^{mn}\p_{x_n}g_{nn} - g^{mr}\p_{x_r}g_{nn} \\
		&= 2g^{mn}\p_{x_n}c - g^{mr}\p_{x_r}c ,\\
		2\chr{n}{ij}&= g^{nr}(\p_{x_j}g_{ir} + \p_{x_i}g_{jr} - \p_{x_r}g_{ij}) = g^{nn}(\p_{x_j}g_{in} + \p_{x_i}g_{jn} - \p_{x_n}g_{ij}) \\
		&= c^{-1}(\p_{x_j}g_{in} + \p_{x_i}g_{jn} - \p_{x_n}g_{ij}) ,\\
		2\chr{m}{nj}&= g^{mr}(\p_{x_j}g_{nr} + \p_{x_n}g_{jr}-\p_{x_r}g_{nj}) ,\\
		2\chr{n}{nj}&= g^{nr}(\p_{x_j}g_{nr} + \p_{x_n}g_{jr}-\p_{x_r}g_{nj}) = g^{nn}(\p_{x_j}g_{nn} + \p_{x_n}g_{jn} - \p_{x_n}g_{nj}) \\
		&= c^{-1}\p_{x_j}c.
	\end{align*}
	Inserting these into \eqref{gen_minsurf_p1} we have
	\begin{align*}
		&c^2|\nabla_gf(x)|_g^2\Delta_gf(x)\\\notag
		&= c\left(1 + |\nabla_{\hat{g}}u|_{\hat{g}}^2\right)(- \Delta_{g}u - \half g^{kl}c^{-1}(\p_{x_l}g_{kn} + \p_{x_k}g_{ln} - \p_{x_n}g_{kl}))\\\notag
		&= \left(1 + |\nabla_{\hat{g}}u|_{\hat{g}}^2\right)(- c\Delta_{g}u - \half (g^{nl}\p_{x_l}c + g^{kn}\p_{x_k}c - g^{kl}\p_{x_n}g_{kl}))\\\notag
		&= \left(1 + |\nabla_{\hat{g}}u|_{\hat{g}}^2\right)(- c\Delta_{g}u - \half c^{-1}(2 - n)\p_{x_n}c).
	\end{align*}
	For calculating the linearizations in Section \ref{section linearizations} it is useful to transform $\Delta_{g}u$ into $\Delta_{\hat{g}}u$:
	\begin{align*}
		\Delta_{g}u&= g^{ij}\left(\p_{x_ix_j}^2u - \half g^{mr}(\p_{x_j}g_{ir} + \p_{x_i}g_{jr} - \p_{x_r}g_{ij})\p_{x_m}u\right) \\
		&= c^{-1}\hat{g}^{ij}\p_{x_ix_j}^2u - \frac{1}{2c}\sum_{m,r=1}^{n-1}\hat{g}^{mr}g^{ij}(\p_{x_j}g_{ir} + \p_{x_i}g_{jr} - \p_{x_r}g_{ij})\p_{x_m}u \\
		&= c^{-1}\hat{g}^{ij}\p_{x_ix_j}^2u - \frac{1}{2c}\sum_{m,r=1}^{n-1}\hat{g}^{mr}\left(\hat{g}^{ij}((\p_{x_j}\hat{g}_{ir} + \p_{x_i}\hat{g}_{jr} - \p_{x_r}\hat{g}_{ij})) + \frac{2-n}{c}\p_{x_r}c\right)\p_{x_m}u \\
		&=c^{-1}\Delta_{\hat{g}}u - \frac{2-n}{2c^2}\hat{g}^{mr}\p_{x_r}c\p_{x_m}u.
	\end{align*}
	Thus
	\begin{align*}
		&c^2|\nabla_gf(x)|_g^2\Delta_gf(x)\\\notag
		&= \left(1 + |\nabla_{\hat{g}}u|_{\hat{g}}^2\right)\left(- \Delta_{\hat{g}}u + \frac{2-n}{2c}\hat{g}^{mr}\p_{x_r}c\p_{x_m}u - \half c^{-1}(2 - n)\p_{x_n}c\right).
	\end{align*}
	Now the second part \eqref{gen_minsurf_p2} becomes
	\begin{align*}
		&c^2\left(\nabla_g^2f(x)\right)\left(\nabla_g f(x),\nabla_g f(x)\right) \\
		&= - \half g^{mr}\p_{x_r}c\p_{x_m}u - \half g^{nn}\p_{x_n}c \\
		&- (g^{mr}(\p_{x_j}g_{nr} + \p_{x_n}g_{jr}-\p_{x_r}g_{nj})\p_{x_m}u - c^{-1}\p_{x_j}c)(\nabla_{\hat{g}}u)^j \\
		&- \nabla_{g}^2u(\nabla_{\hat{g}}u,\nabla_{\hat{g}}u) - \half c^{-1}(\p_{x_j}g_{in} + \p_{x_i}g_{jn} - \p_{x_n}g_{ij})(\nabla_{\hat{g}}u)^i(\nabla_{\hat{g}}u)^j \\
		&= -\half c^{-1}\sum_{m,r=1}^{n-1}\hat{g}^{mr}\p_{x_r}c\p_{x_m}u - \half c^{-1}\p_{x_n}c \\
		&- \sum_{m,r=1}^{n-1}c^{-1}\hat{g}^{mr}\hat{g}_{jr}\p_{x_n}c\p_{x_m}u(\nabla_{\hat{g}}u)^j + c^{-1}\p_{x_j}c(\nabla_{\hat{g}}u)^j \\
		&- \nabla_{g}^2u(\nabla_{\hat{g}}u,\nabla_{\hat{g}}u) + \half c^{-1}\p_{x_n}c\hat{g}_{ij}(\nabla_{\hat{g}}u)^i(\nabla_{\hat{g}}u)^j \\
		&=  \half c^{-1}\sum_{m,r=1}^{n-1}\hat{g}^{mr}\p_{x_r}c\p_{x_m}u - \half c^{-1}\p_{x_n}c(1 + |\nabla_{\hat{g}}u|^2_{\hat{g}}) - \nabla_{g}^2u(\nabla_{\hat{g}}u,\nabla_{\hat{g}}u).
	\end{align*}
	As before we would like to modify $\nabla_{g}^2u(\nabla_{\hat{g}}u,\nabla_{\hat{g}}u)$ to have $\nabla_{\hat{g}}^2u(\nabla_{\hat{g}}u,\nabla_{\hat{g}}u)$:
	\begin{align*}
		&\nabla_{g}^2u(\nabla_{\hat{g}}u,\nabla_{\hat{g}}u) \\
		&= (\p_{x_ix_j}^2u - \half g^{mr}(\p_{x_i}g_{rj} + \p_{x_j}g_{ri} - \p_{x_r}g_{ji})\p_{x_m}u)\hat{g}^{ia}\p_{x_a}u\hat{g}^{jb}\p_{x_b}u.
	\end{align*}
	Since $\p_{x_n}u=0$, all the indices run from $1$ to $n-1$. Hence
	\begin{align*}
		&\nabla_{g}^2u(\nabla_{\hat{g}}u,\nabla_{\hat{g}}u) \\
		&= (\p_{x_ix_j}^2u - \frac{1}{2c} \hat{g}^{mr}(\p_{x_i}c\hat{g}_{rj} + c\p_{x_i}\hat{g}_{rj} + \p_{x_j}c\hat{g}_{ri} + c\p_{x_j}\hat{g}_{ri} \\
		&- \p_{x_r}c\hat{g}_{ij} - c\p_{x_r}\hat{g}_{ij})\p_{x_m}u)\hat{g}^{ia}\p_{x_a}u\hat{g}^{jb}\p_{x_b}u \\
		&= (\p_{x_ix_j}^2u - \chrhat{m}{ij}\p_{x_m}u - \frac{1}{2c}(\delta^m_j\p_{x_i}c + \delta^m_i\p_{x_j}c - \hat{g}^{mr}\hat{g}_{ij}\p_{x_r}c)\p_{x_m}u)\hat{g}^{ia}\p_{x_a}u\hat{g}^{jb}\p_{x_b}u \\
		&= \nabla_{\hat{g}}^2u(\nabla_{\hat{g}}u,\nabla_{\hat{g}}u) - \frac{1}{2c}(\hat{g}^{ia}\p_{x_i}c\p_{x_a}u\hat{g}^{mb}\p_{x_m}u\p_{x_b}u + \hat{g}^{jb}\p_{x_j}c\p_{x_b}u\hat{g}^{mb}\p_{x_m}u\p_{x_a}u \\
		&- \hat{g}^{mr}\p_{x_r}c\p_{x_m}u\hat{g}^{jb}\p_{x_j}u\p_{x_b}u) \\
		&= \nabla_{\hat{g}}^2u(\nabla_{\hat{g}}u,\nabla_{\hat{g}}u) - \frac{1}{2c}\hat{g}^{ia}\p_{x_i}c\p_{x_a}u|\nabla_{\hat{g}}u|_{\hat{g}}^2,
	\end{align*}
	where $2\chrhat{m}{ij}=\hat{g}^{mr}(\p_{x_j}\hat{g}_{ir} + \p_{x_i}\hat{g}_{jr}-\p_{x_r}\hat{g}_{ij})$.

	Putting these together gives
	\begin{align*}
		0&= c^2|\nabla_gf(x)|_g^2\Delta_gf(x) - c^2\left(\nabla_g^2f(x)\right)\left(\nabla_g f(x),\nabla_g f(x)\right) \\
		&= \left(- \Delta_{\hat{g}}u + \frac{1-n}{2c}\hat{g}^{mr}\p_{x_r}c\p_{x_m}u + \frac{n-1}{2c}\p_{x_n}c\right)(1 + |\nabla_{\hat{g}}u|^2_{\hat{g}}) \\
		&+ \nabla_{\hat{g}}^2u(\nabla_{\hat{g}}u,\nabla_{\hat{g}}u).
	\end{align*}
	Thus $\text{Graph}_u$ is a minimal surface if and only if the function $u$ satisfies the following minimal surface equation
	\begin{align}\label{mse1}
		&\left(- \Delta_{\hat{g}}u + \frac{1-n}{2c(x',u(x'))}\hat{g}^{mr}(x')\p_{x_r}c(x',u(x'))\p_{x_m}u + \frac{n-1}{2c(x',u(x'))}\p_{x_n}c(x',u(x'))\right)(1 + |\nabla_{\hat{g}}u|^2_{\hat{g}}) \\\notag
		&+ \nabla_{\hat{g}}^2u(\nabla_{\hat{g}}u,\nabla_{\hat{g}}u) = 0
	\end{align}
	for all $x'\in\Om$.
	
	We will modify this a bit in order to see the more familiar Euclidean version of the equation directly. For this, we start with the divergence of  a vector field defined in local coordinates as
	\begin{equation*}
		\dive_ga=\sum_{i=1}^n\left(\p_{x_i}a^i + \sum_{j=1}^n a^j\chr{i}{ij}\right).
	\end{equation*}
	Let us denote $\eta:=\left(1+|\nabla_{g}u|_g^2\right)^{\half}$ and expand the following for a CTA metric and a function $u\colon \Om\to\R$ as before:
	\begin{align*}
		&\dive_g\left(\frac{\nabla_{g}u}{\eta}\right) \\
		&= \sum_{i=1}^n\left(\p_{x_i}\left(\frac{(\nabla_{g}u)^i}{\eta}\right) + \sum_{j=1}^n \frac{(\nabla_{g}u)^j}{\eta}\chr{i}{ij}\right) \\
		&= \sum_{i=1}^{n-1}\frac{\p_{x_i}g^{ik}\p_{x_k}u + g^{ik}\p_{x_ix_k}^2u}{\eta} \\
		&- \half \frac{(\nabla_{g}u)^i}{\eta^3} \left(\p_{x_i}g^{ab}\p_{x_a}u\p_{x_b}u + g^{ab}\left(\p_{x_ix_a}^2u\p_{x_b}u + \p_{x_a}u\p_{x_ix_b}^2u\right)\right) \\
		&+ \sum_{j=1}^{n-1}\frac{(\nabla_{g}u)^j}{\eta}\left(\chr{i}{ij} + \chr{n}{nj}\right).
	\end{align*}
	Now we will use that $\p_{x_l}g^{ik}=-(\chr{i}{ml}g^{mk} + \chr{k}{ml}g^{im})$ to have
	\begin{align*}
		&\dive_g\left(\frac{\nabla_{g}u}{\eta}\right) \\
		&= \sum_{i=1}^{n-1}\frac{-(\chr{i}{il}g^{lk} + \chr{k}{il}g^{il})\p_{x_k}u + g^{ik}\p_{x_ix_k}^2u}{\eta} \\
		&- \half \frac{(\nabla_{g}u)^i}{\eta^3} \left(-(\chr{a}{il}g^{lb} + \chr{b}{il}g^{al})\p_{x_a}u\p_{x_b}u + g^{ab}\left(\p_{x_ix_a}^2u\p_{x_b}u + \p_{x_a}u\p_{x_ix_b}^2u\right)\right) \\
		&+ \sum_{j=1}^{n-1}\frac{(\nabla_{g}u)^j}{\eta}\left(\chr{i}{ij} + \chr{n}{nj}\right) \\
		&= \frac{\Delta_{g}u}{\eta} - \sum_{i,l=1}^{n-1} \frac{(\nabla_{g}u)^l}{\eta}\chr{i}{il} \\
		&- \half \frac{(\nabla_{g}u)^i}{\eta^3} \left(-(\nabla_{g}u)^l\chr{a}{il}\p_{x_a}u - (\nabla_{g}u)^l\chr{b}{il}\p_{x_b}u + (\nabla_{g}u)^a\p_{x_ix_a}^2u + (\nabla_{g}u)^b\p_{x_ix_b}^2u\right) \\
		&+ \sum_{j=1}^{n-1}\frac{(\nabla_{g}u)^j}{\eta}\left(\chr{i}{ij} + \chr{n}{nj}\right).
	\end{align*}
	After renaming some of the indices, using the definition of the Riemannian Hessian of a function and $\chr{n}{nj}=\half c^{-1}\p_{x_j}c$ we get
	\begin{equation}\label{div_g}
		\dive_g\left(\frac{\nabla_{g}u}{\eta}\right) = \frac{\Delta_{g}u}{\eta} - \frac{\nabla_{g}^2u(\nabla_{g}u,\nabla_{g}u)}{\eta^3} + \sum_{j=1}^{n-1}\frac{(\nabla_{g}u)^j}{2c\eta}\p_{x_j}c.
	\end{equation}
	Let us now go back to \eqref{mse1} and modify it to have \eqref{div_g} visible. Now
	\begin{align*}
	 	&\left(- \Delta_{\hat{g}}u + \frac{1-n}{2c}\hat{g}^{mr}\p_{x_r}c\p_{x_m}u + \frac{n-1}{2c}\p_{x_n}c\right)(1 + |\nabla_{\hat{g}}u|^2_{\hat{g}}) \\
	 	&+ \nabla_{\hat{g}}^2u(\nabla_{\hat{g}}u,\nabla_{\hat{g}}u) \\
	 	&= \left(- \Delta_{\hat{g}}u + \frac{2-n}{2c}\hat{g}^{mr}\p_{x_r}c\p_{x_m}u + \frac{n-1}{2c}\p_{x_n}c\right)(1 + |\nabla_{\hat{g}}u|^2_{\hat{g}}) \\
	 	&- \half \hat{g}^{mr}\p_{x_r}c\p_{x_m}u(1 + |\nabla_{\hat{g}}u|^2_{\hat{g}}) + \nabla_{\hat{g}}^2u(\nabla_{\hat{g}}u,\nabla_{\hat{g}}u) \\
	 	&= \left(- c\Delta_{g}u + \frac{n-1}{2c}\p_{x_n}c\right)(1 + |\nabla_{\hat{g}}u|^2_{\hat{g}}) - \half \hat{g}^{mr}\p_{x_r}c\p_{x_m}u + \nabla_{g}^2u(\nabla_{\hat{g}}u,\nabla_{\hat{g}}u).
	\end{align*}
 	Then, we will multiply both sides of \eqref{mse1} by $c^{-2}$ to get
	\begin{align*}
	 	0 &= \left(- \Delta_{g}u + \frac{n-1}{2c^2}\p_{x_n}c\right)(c^{-1} + |\nabla_{g}u|^2_{g}) - \half c^{-3}\hat{g}^{mr}\p_{x_r}c\p_{x_m}u + \nabla_{g}^2u(\nabla_{g}u,\nabla_{g}u) \\
	 	&= - \Delta_{g}u|\nabla_{g}u|^2_{g} - c^{-1}\Delta_{g}u + \Delta_{g}u - \Delta_{g}u + \frac{(\nabla_{g}u)^j}{2c}\p_{x_j}c\eta^2 - \frac{(\nabla_{g}u)^j}{2c}\p_{x_j}c\eta^2 \\
	 	& + \frac{n-1}{2c^3}\p_{x_n}c(1 + |\nabla_{\hat{g}}u|^2_{\hat{g}}) - \half c^{-3}(\nabla_{g}u)^r\p_{x_r}c + \nabla_{g}^2u(\nabla_{g}u,\nabla_{g}u) \\
	 	&= - \Delta_g u\eta^2 + \nabla_{g}^2u(\nabla_{g}u,\nabla_{g}u) - \frac{(\nabla_{g}u)^j}{2c}\p_{x_j}c\eta^2 + \Delta_g u(1-c^{-1}) \\
	 	&+ \frac{n-1}{2c^3}\p_{x_n}c(1 + |\nabla_{\hat{g}}u|^2_{\hat{g}}) + \frac{1}{2c}(\nabla_{g}u)^j\p_{x_j}c(\eta^2 - c^{-2}).
	\end{align*}
 	Dividing both sides by $\eta^3=\left(1+|\nabla_{g}u|_g^2\right)^{\frac{3}{2}}$ gives
 	\begin{align*}
 		&\dive_g\left(\frac{\nabla_{g}u}{\left(1+|\nabla_{g}u|_g^2\right)^{\half}}\right) \\
 		&+ \frac{\Delta_{g}u(1-c^{-1}) + \frac{1}{2c}(\nabla_{g}u)^j\p_{x_j}c(1 + |\nabla_{g}u|_g^2 - c^{-2}) + \frac{n-1}{2c^3}\p_{x_n}c(1 + |\nabla_{\hat{g}}u|^2_{\hat{g}})}{\left(1+|\nabla_{g}u|_g^2\right)^{\frac{3}{2}}} = 0.
 	\end{align*}
	When $g$ is the Euclidean metric, we will have the familiar Euclidean minimal surface equation. Moreover, when $c=1$ we get the minimal surface equation in \cite{Carstea2022}:
	\begin{equation*}
		\dive_{\hat{g}}\left(\frac{\nabla_{\hat{g}}u}{\left(1+|\nabla_{\hat{g}}u|_{\hat{g}}^2\right)^{\half}}\right)=0.
	\end{equation*}

	\section{First and second order linearizations}\label{section linearizations}

	Let $g$ be as in \eqref{metricg} and define $F\colon \R^{n^2} \to \R,$
	\begin{align}\label{def_F}
		F(x',u,\nabla_g u,\nabla_{g}^2u):=& \left(- \Delta_{\hat{g}}u + \frac{1-n}{2c}\hat{g}^{mr}\p_{x_r}c\p_{x_m}u + \frac{n-1}{2c}\p_{x_n}c\right)(1 + |\nabla_{\hat{g}}u|^2_{\hat{g}}) \\
		&+ \nabla_{\hat{g}}^2u(\nabla_{\hat{g}}u,\nabla_{\hat{g}}u). 
	\end{align} 
	We want to use the well-posedness result from \cite{Nurminen2022} to say that the boundary value problem
	\begin{equation}\label{BVP_main}
		\left\{\begin{array}{ll}
			F(x',u,\nabla_g u,\nabla_g^2u)=0 & \text{in}\,\, \Om \\
			u=f, & \text{on}\,\, \partial\Om
		\end{array} \right.
	\end{equation}
	has a unique small solution. For this we need that $F(x',0,0,0)=0$ which guarantees that $u\equiv0$ is a solution to \eqref{BVP_main} with $f=0$. From \eqref{def_F} we can see that this happens if and only if $\p_{x_n}c(x',0)=0$ for all $x'\in\Om$.
	
	To use the above mentioned well-posedness, we will need to calculate the first linearization of  the equation $F(x',u,\nabla_g u,\nabla_g^2u)=0$. Let us do a formal calculation to get the first linearization. Let $\e=(\e_1,\ldots,\e_m)$ where $\e_k\in\R$ small and assume that $u_{\e}:=u(x,\e)$ depends smoothly on $\e$ and solves $F(x',u_{\e},\nabla_{g}u_{\e},\nabla_{g}^2u_{\e})=0$. We will differentiate this in three parts with respect to $\e_l$ and evaluate at $\e=0$ (We could have calculated the first linearization as in \cite{Nurminen2022}, but it would be a longer calculation). Denote by $u_{\e}^{(l)}:=\p_{\e_l}u_{\e}$. Just to make the calculations a bit nicer we will start with
	\begin{align}\label{first_lin_2}
		&\p_{\e_l}\left(\frac{1-n}{2c} \hat{g}^{mr}\p_{x_r}c\p_{x_m}u_{\e}\right) \\\notag
		&= - \frac{1-n}{2c^2} \p_{x_n}cu_{\e}^{(l)}\hat{g}^{mr}\p_{x_r}c\p_{x_m}u_{\e} + \frac{1-n}{2c} \hat{g}^{mr}(\p_{x_rx_n}^2cu_{\e}^{(l)}\p_{x_m}u_{\e} + \p_{x_r}c\p_{x_m}u_{\e}^{(l)})\\\notag
		&:= Q_l.
	\end{align}
	Now the first term of \eqref{def_F} becomes
	\begin{align}\label{first_lin_1}
		&\p_{\e_l}\left(\left(- \Delta_{\hat{g}}u_{\e} + \frac{1-n}{2c}\hat{g}^{mr}\p_{x_r}c\p_{x_m}u_{\e} + \frac{n-1}{2c}\p_{x_n}c\right)(1 + |\nabla_{\hat{g}}u_{\e}|^2_{\hat{g}})\right) \\\notag
		&= \left(- \Delta_{\hat{g}}u_{\e}^{(l)} - \frac{n-1}{2c^{-2}}(\p_{x_n}c)^2u_{\e}^{(l)} + \frac{n-1}{2c}\p_{x_n}^2cu_{\e}^{(l)} + Q_l\right)(1 + |\nabla_{\hat{g}}u_{\e}|^2_{\hat{g}}) \\\notag
		&+ \left(- \Delta_{\hat{g}}u_{\e} + \frac{1-n}{2c}\hat{g}^{mr}\p_{x_r}c\p_{x_m}u_{\e} + \frac{n-1}{2c}\p_{x_n}c\right)\hat{g}^{ij}(\p_{x_i}u_{\e}^{(l)}\p_{x_j}u_{\e} + \p_{x_i}u_{\e}\p_{x_j}u_{\e}^{(l)}).
	\end{align}
	The second term is
	\begin{align}\label{first_lin_3}
		&\p_{\e_l}\left(\nabla_{\hat{g}}^2u_{\e}(\nabla_{\hat{g}}u_{\e},\nabla_{\hat{g}}u_{\e})\right) \\\notag
		&= \p_{\e_l}\left((\p_{x_ix_j}^2u_{\e} - \chrhat{m}{ij}\p_{x_m}u_{\e})\hat{g}^{ia}\p_{x_a}u_{\e}\hat{g}^{jb}\p_{x_b}u_{\e}\right) \\\notag
		&= (\p_{x_ix_j}^2u_{\e}^{(l)} - \chrhat{m}{ij}\p_{x_m}u_{\e}^{(l)})\hat{g}^{ia}\p_{x_a}u_{\e}\hat{g}^{jb}\p_{x_b}u_{\e} \\\notag
		&+ (\p_{x_ix_j}^2u_{\e} - \chrhat{m}{ij}\p_{x_m}u_{\e})\hat{g}^{ia}\hat{g}^{jb}(\p_{x_a}u_{\e}^{(l)}\p_{x_b}u_{\e} + \p_{x_a}u_{\e}\p_{x_b}u_{\e}^{(l)})\\\notag
		&:= P_1 + P_2.
	\end{align}
	where $P_1$ and $P_2$ are for future reference.
	When evaluating at $\e=0$ we have $u_0=0$, also $\nabla_{g}u_{0},\nabla_{g}^2u_{0}$ vanish,  and thus when combining the above, denoting $v^l:=u_{\e}^{(l)}|_{\e=0}$, 
	\begin{align*}
		&\p_{\e_l}F(x',u_{\e},\nabla_{g}u_{\e},\nabla_{g}^2u_{\e})|_{\e=0} \\
		&= - \Delta_{\hat{g}}v^l + \frac{n-1}{2c(x',0)}\p_{x_n}^2c(x',0)v^l + \frac{1-n}{2c(x',0)} \sum_{i,j=1}^{n-1}\hat{g}^{ij}\p_{x_i}c(x',0)\p_{x_j}v^l = 0.
	\end{align*}
	Now this first linearization satisfies the assumptions of \cite[Proposition 2.1]{Nurminen2022} and thus we have the existence and uniqueness of small solutions to \eqref{BVP_main}.
	
	Next we will calculate the second linearization for the equation $F(x',u,\nabla_g u,\nabla_g^2u)=0$. Denote $u_{\e}^{(kl)}:=\p_{\e_k\e_l}^2u_{\e}$. We will differentiate \eqref{first_lin_2}-\eqref{first_lin_3} with respect to $\e_k$, $k\neq l$, starting with \eqref{first_lin_2}:
	\begin{align*}
		&\p_{\e_k}\p_{\e_l}\left(\frac{1-n}{2c} \hat{g}^{mr}\p_{x_r}c\p_{x_m}u_{\e}\right) \\
		&=  \frac{1-n}{c^3}(\p_{x_n}c)^2u_{\e}^{(k)}u_{\e}^{(l)} \hat{g}\p_{x_r}c\p_{x_m}u_{\e} - \frac{1-n}{2c^2}\p_{x_n}^2cu_{\e}^{(k)}u_{\e}^{(l)} \hat{g}\p_{x_r}c\p_{x_m}u_{\e} \\
		&- \frac{1-n}{2c^2}\p_{x_n}cu_{\e}^{(kl)} \hat{g}\p_{x_r}c\p_{x_m}u_{\e} \\
		&- \frac{1-n}{2c^2} \p_{x_n}cu_{\e}^{(l)}\hat{g}^{mr}(\p_{x_rx_n}^2cu_{\e}^{(k)}\p_{x_m}u_{\e} + \p_{x_r}c\p_{x_m}u_{\e}^{(k)}) \\
		&- \frac{1-n}{2c^2} \p_{x_n}c\hat{g}^{mr}(\p_{x_rx_n}^2cu_{\e}^{(l)}\p_{x_m}u_{\e} + \p_{x_r}c\p_{x_m}u_{\e}^{(l)}) \\
		&+ \frac{1-n}{2c}\hat{g}^{mr}\big(\p_{x_r}\p_{x_n}^2cu_{\e}^{(k)}u_{\e}^{(l)}\p_{x_m}u_{\e} + \p_{x_rx_n}^2c(u_{\e}^{(kl)}\p_{x_m}u_{\e} + u_{\e}^{(l)}\p_{x_m}u_{\e}^{k}) \\
		&+ \p_{x_rx_n}cu_{\e}^{(k)}\p_{x_m}u_{\e}^{(l)} + \p_{x_r}c\p_{x_m}u_{\e}^{(kl)}\big).
	\end{align*}
	Then we differentiate \eqref{first_lin_1} to have
	\begin{align*}
		&\p_{\e_k}\p_{\e_l}\left(\left(- \Delta_{\hat{g}}u_{\e} + \frac{1-n}{2c}\hat{g}^{mr}\p_{x_r}c\p_{x_m}u_{\e} + \frac{n-1}{2c}\p_{x_n}c\right)(1 + |\nabla_{\hat{g}}u_{\e}|^2_{\hat{g}})\right) \\
		&= \bigg(- \Delta_{\hat{g}}u_{\e}^{(kl)} + \frac{n-1}{c^{-3}}(\p_{x_n}c)^3u_{\e}^{(k)}u_{\e}^{(l)} - \frac{n-1}{2c^{-2}}\left(2\p_{x_n}c\p_{x_n}^2cu_{\e}^{(k)}u_{\e}^{(l)} + (\p_{x_n}c)^2u_{\e}^{(kl)}\right) \\
		&- \frac{n-1}{2c^{-2}}\p_{x_n}c\p_{x_n}^2cu_{\e}^{(k)}u_{\e}^{(l)} + \frac{n-1}{2c}\left(\p_{x_n}^3cu_{\e}^{(k)}u_{\e}^{(l)} + \p_{x_n}^2cu_{\e}^{(kl)}\right) + \p_{\e_k}Q_l\bigg)(1+|\nabla_{\hat{g}}u_{\e}|_{\hat{g}}^2) \\
		&+ \left(- \Delta_{\hat{g}}u_{\e}^{(l)} - \frac{n-1}{2c^{-2}}(\p_{x_n}c)^2u_{\e}^{(l)} + \frac{n-1}{2c}\p_{x_n}^2cu_{\e}^{(l)} + Q_l\right)\hat{g}^{ij}(\p_{x_i}u_{\e}^{(k)}\p_{x_j}u_{\e} + \p_{x_i}u_{\e}\p_{x_j}u_{\e}^{(k)}) \\
		&+ \left(- \Delta_{\hat{g}}u_{\e}^{(k)} - \frac{n-1}{2c^{-2}}(\p_{x_n}c)^2u_{\e}^{(k)} + \frac{n-1}{2c}\p_{x_n}^2cu_{\e}^{(k)}  + Q_k\right)\hat{g}^{ij}(\p_{x_i}u_{\e}^{(l)}\p_{x_j}u_{\e} + \p_{x_i}u_{\e}\p_{x_j}u_{\e}^{(l)}) \\
		&+ \left(- \Delta_{\hat{g}}u_{\e} + \frac{1-n}{2c}\hat{g}^{mr}\p_{x_r}c\p_{x_m}u_{\e} + \frac{n-1}{2c}\p_{x_n}c\right)\hat{g}^{ij}(\p_{x_i}u_{\e}^{(kl)}\p_{x_j}u_{\e} + \p_{x_i}u_{\e}^{(l)}\p_{x_j}u_{\e}^{(k)} \\
		&+ \p_{x_i}u_{\e}^{(k)}\p_{x_j}u_{\e}^{(l)} + \p_{x_i}u_{\e}\p_{x_j}u_{\e}^{(kl)}).
	\end{align*}

	Now we calculate the derivative of \eqref{first_lin_3} in two parts, starting with $P_1$:
	\begin{align*}
		&\p_{\e_k}P_1 \\
		&= (\p_{x_ix_j}^2u_{\e}^{(kl)} - \chrhat{m}{ij}\p_{x_m}u_{\e}^{(kl)})\hat{g}^{ia}\p_{x_a}u_{\e}\hat{g}^{ja}\p_{x_b}u_{\e} \\
		&+ (\p_{x_ix_j}^2u_{\e}^{(l)} - \chrhat{m}{ij}\p_{x_m}u_{\e}^{(l)})\hat{g}^{ia}\hat{g}^{jb}(\p_{x_a}u_{\e}^{(k)}\p_{x_b}u_{\e} + \p_{x_a}u_{\e}\p_{x_b}u_{\e}^{(k)}).
	\end{align*}
	For $P_2$ we get
	\begin{align*}
		&\p_{\e_k}P_2\\
		&= (\p_{x_ix_j}^2u_{\e}^{(k)} - \chrhat{m}{ij}\p_{x_m}u_{\e}^{(k)})\hat{g}^{ia}\hat{g}^{jb}(\p_{x_a}u_{\e}^{(l)}\p_{x_b}u_{\e} + \p_{x_a}u_{\e}\p_{x_b}u_{\e}^{(l)}) \\
		&+ (\p_{x_ix_j}u_{\e} - \chrhat{m}{ij}\p_{x_m}u_{\e})\hat{g}^{ia}\hat{g}^{jb}(\p_{x_a}u_{\e}^{(kl)}\p_{x_b}u_{\e} + \p_{x_a}u_{\e}^{(l)}\p_{x_b}u_{\e}^{(k)} \\
		&+ \p_{x_a}u_{\e}^{(k)}\p_{x_b}u_{\e}^{(l)} + \p_{x_a}u_{\e}\p_{x_b}u_{\e}^{(kl)}).
	\end{align*}
	Evaluating the above second derivatives at $\e=0$, denoting $w^{kl}:=u_{\e}^{(kl)}|_{\e=0}$, gives the second linearization of the equation $F(x',u,\nabla_g u,\nabla_g^2u)=0$:
	\begin{align*}
		&\p_{\e_k\e_l}^2F(x',u_{\e},\nabla_{g}u_{\e},\nabla_{g}^2u_{\e})|_{\e=0} \\
		&= - \Delta_{\hat{g}}w^{kl} + \frac{n-1}{2c(x',0)}\left(\p_{x_n}^3c(x',0)v^kv^l + \p_{x_n}^2c(x',0)w^{kl}\right) \\
		&+ \frac{1-n}{2c(x',0)}\sum_{i,j=1}^{n-1}\hat{g}^{ij}\p_{x_i}c(x',0)\p_{x_j}w^{kl}=0.
	\end{align*}
	Here we used that $u_0=0$, $\nabla_{g}u_{0},\nabla_{g}^2u_{0}$ vanish, $\p_{x_n}c(x',0)=0$ and that $\p_{x_kx_n}c(x',0)=0$ for $k=1,\ldots,n-1$.

	\section{Proofs of Theorems \ref{main thm} and \ref{main thm_partial}}\label{section proof of main}

	Before going to the proof of Theorem \ref{main thm}, we state two lemmas that will be used also in the proof of Theorem \ref{main thm_partial}. The first lemma says that the products of two solutions to an advection diffusion equation form a complete set in $L^1$. For this we note that equations of the form $\Delta_{g}u + Xu = 0$, for a smooth real valued vector field $X$, can be written in the form of a magnetic Schrödinger equation. In local coordinates we have
	\begin{equation*}
		L_{g,A,q}u = -|g|^{-\half}\left(\p_{x_j} + iA_j\right)\left(|g|^{\half}g^{jk}\left(\p_{x_k} + iA_k\right)u\right) + qu = 0
	\end{equation*}
	for $A=\frac{iX}{2}$ and $q=\frac{1}{4} g(X,X) - \half\dive_g(X)$.
	\begin{lemma}\label{lemma_density}
		Let $(\Om,g)$, $\Om\subset\Rnn$, be a smooth Riemannian manifold with boundary. We have two cases:
		\begin{enumerate}
			\item When $n=2$ let $\Gamma\subset\p\Om$ be a nonempty open set, let $f\in C^{\infty}(\bar{\Om})$ be such that $f|_{\Gamma}=0$. Assume that
			\begin{equation}\label{int_id_lemma}
				\int_{\Om}fu_1u_2\,dV_{g} = 0
			\end{equation}
			for all $u_j$ solving $L_{g,A,q}u_j=0$ in $\Om$ with $u_j|_{\p\Om\setminus\Gamma}=0$. Then $f\equiv0$ in $\Om$.
			\item When $n>2$ let $g$ be admissible, $f\in C^{\infty}(\bar{\Om})$ be such that $f|_{\p\Om}=0$. Assume that \eqref{int_id_lemma} holds for all $u_j$ solving $L_{g,A,q}u_j=0$ in $\Om$. Then $f\equiv0$ in $\Om$.
		\end{enumerate}
	\end{lemma}
	\begin{proof}
		$n=2$: The proof is written in \cite[Section 7.3]{Tzou2017} where the author considers identifying a zeroth order term.
		
		$n>2$: This can be read from the proof of \cite[Theorem 1.7]{Ferreira2009b}, more precisely from the part where they prove that the two potentials $q_1$ and $q_2$ agree in $M$ (using the notation of the referred article).
	\end{proof}

	The second lemma states that if we know the partial DN maps corresponding to $\Delta_{g}u+X_ju=0$ for $j=1,2$ and $X_1=X_2$ on $\Gamma\subset\p\Om$, $\Gamma\neq\emptyset$, then $X_1=X_2$ in $\Om$. The DN map in question is defined as $\Lambda_{X}^{\Gamma}\colon C^{k,\alpha}(\Gamma)\to C^{k-1,\alpha}(\Gamma),$
	\begin{equation*}
		\Lambda_{X}^{\Gamma}f=\p_{\nu}u|_{\Gamma}.
	\end{equation*}
	
	\begin{lemma}\label{lemma_first_lin_identification}
		Let $(\R^2,g)$ be a Riemannian manifold, $\Om\subset\R^2$ be simply connected with smooth boundary and $\Gamma\subset\p\Om$ be an open nonempty set. Let $X_j$ be smooth vector fields for $j=1,2$ so that $X_1|_{\Gamma}=X_2|_{\Gamma}$. If $\Lambda_{X_1}^{\Gamma}f=\Lambda_{X_2}^{\Gamma}f$ for all $f\in C^{k,\alpha}(\Gamma)$, then $X_1=X_2$ in $\Om$.
	\end{lemma}
	\begin{proof}
		By \cite[Theorem 1.1]{Tzou2017} we have $iA_1-iA_2=\theta^{-1}d\theta$ and $q_1=q_2$ in $\Om$, where
		\begin{equation*}
			A_j=\frac{iX_j}{2},\quad q_j=\frac{1}{4}g(X_j,X_j) - \half\dive_g(X_j)
		\end{equation*}
		and $\theta$ is a nonvanishing function with $\theta|_{\Gamma}=1$.
		Notice that $d(\theta^{-1}d\theta)=d(\theta^{-1})\wedge d\theta + \theta^{-1}\wedge d^2\theta = 0$ which implies that $dX_1=dX_2$ and by the Poincaré Lemma we have a function $\varphi\in C^{\infty}(\bar{\Om})$ such that
		\begin{equation*}
			X_1-X_2=\nabla_g\varphi.
		\end{equation*}
		Using that $X_1|_{\Gamma}=X_2|_{\Gamma}$ gives $0=(X_1-X_2)\cdot\nu|_{\Gamma}=\p_{\nu}\varphi|_{\Gamma}$ and $\varphi|_{\Gamma}=0$ (actually we have that $\varphi|_{\p\Om}$ is a constant, but we can subtract this constant to obtain $\varphi|_{\p\Om}=0$). Let us see what kind of an equation $\varphi$ satisfies. Now $X_2=X_1-\nabla_g\varphi$ and plugging this into $q_1=q_2$ gives
		\begin{align*}
			&\frac{1}{4}g(X_1,X_1) - \half\dive_g(X_1) \\
			&= \frac{1}{4}g^{kl}\left(X_{1,k}-(\nabla_g\varphi)_k\right)\left(X_{1,l}-(\nabla_g\varphi)_l\right) - \half\dive_g(X_1) + \half\dive_g(\nabla_g\varphi) \\
			&= \frac{1}{4}g(X_1,X_1) - \half g(X_1,\nabla_g\varphi) + \frac{1}{4}g(\nabla_g\varphi, \nabla_g\varphi) - \half\dive_g(X_1) + \half \Delta_g\varphi.
		\end{align*}
		Thus $\varphi$ solves
		\begin{equation}\label{BVP_first_lin}
			\left\{\begin{array}{ll}
				\Delta_g\varphi - g(X_1,\nabla_g\varphi) + \half |\nabla_{g}\varphi|_g^2 = 0 & \text{in}\,\, \Om \\
				\varphi = 0 & \text{on}\,\, \Gamma\\
				\p_{\nu}\varphi = 0 & \text{on}\,\, \Gamma.
			\end{array} \right.
		\end{equation}
		From this we get that
		\begin{align*}
			|\Delta_{g}\varphi|=|g(X_1,\nabla_g\varphi) - \half |\nabla_{g}\varphi|_g^2|\leq |g(X_1,\nabla_g\varphi)| + |\half |\nabla_{g}\varphi|_g^2| \leq C|\nabla_{g}\varphi|_g\leq C(|\nabla_{g}\varphi|_g + |\varphi|)
		\end{align*}
		since $|\nabla_{g}\varphi|_g\leq M$ for some $M>0$.
		Then by the unique continuation principle for local Cauchy data $\varphi\equiv0$ in $\Om$ (see \cite[Theorem B.$1$]{Kenig2011}) and hence $X_1=X_2$.
	\end{proof}
	
	\begin{proof}[Proof Theorem \ref{main thm}]
		The assumptions of \cite[Proposition 2.1]{Nurminen2022} are satisfied and thus \eqref{BVP_main} has a unique small solution $u_f\in C^{s}(\bar{\Om})$. Let $\e=(\e_1,\ldots,\e_N)$ where $\e_k\in \R$ small enough so that \eqref{BVP_main} has a unique small solution for $f_{\e}:=\sum_{k=1}^{N}f_k$, where $f_k\in C^s(\p\Om)$ with $\norm{f_k}_{C^s(\p\Om)}< \delta$, $\delta>0$. Denote by $u:=u(x,\e)$ and $\tilde{u}:=\tilde{u}(x,\e)$ the unique small solutions to \eqref{BVP_main} with the mterics $g$ and $\tilde{c}g$ respectively. Then by \cite[Proposition 2.1]{Nurminen2022} the solutions $u$ and $\tilde{u}$ depend smoothly on $\e$. Hence we can calculate the first and second order linearizations of \eqref{BVP_main} and the linearizations of the DN map \eqref{DN_map_intro}.
		
		In section \ref{section linearizations} we already calculated the first linearization of $F(x',u,\nabla_g u,\nabla_g^2u)=0$ and since we also assume that $\p_{x_n}^2c(x',0)=0$ we have that the first linearization of \eqref{BVP_main} is
		\begin{equation}\label{BVP_first_lin}
			\left\{\begin{array}{ll}
				 - \Delta_{\hat{g}}v^l + \frac{1-n}{2c(x',0)} \sum_{i,j=1}^{n-1}\hat{g}^{ij}\p_{x_i}c(x',0)\p_{x_j}v^l = 0 & \text{in}\,\, \Om \\
				v^l=f_l, & \text{on}\,\, \partial\Om,
			\end{array} \right.
		\end{equation}
		where $v^l=\p_{\e_l}u|_{\e=0}$.
		Also, the first linearization of the DN map is $(D\Lambda_g)_0\colon C^s(\p\Om)\to C^{s-1}(\p\Om), f\mapsto \p_{\nu}v^l|_{\p\Om}$. Now \eqref{BVP_first_lin} can be written in the form of an advection diffusion equation $ - (\Delta_{\hat{g}} - X)v^l=0$ for the real vector field $X$ such that 
		\begin{equation*}
			Xh = \frac{1-n}{2c(x',0)}\hat{g}^{ij}\p_{x_i}c(x',0)\p_{x_j}h.
		\end{equation*}	
		For the metric $\tilde{c}g$ the first linearization of $F(x',\tilde{u},\nabla_{\tilde{c}g} \tilde{u},\nabla_{\tilde{c}g}^2\tilde{u})=0$ is
		\begin{equation}\label{BVP_first_lin_tilde}
			\left\{\begin{array}{ll}
				- \Delta_{\hat{g}}\tilde{v}^l + \frac{1-n}{2c(x',0)} \sum_{i,j=1}^{n-1}\hat{g}^{ij}\p_{x_i}c(x',0)\p_{x_j}\tilde{v}^l \phantom{\bigg|}\\
				+ \frac{1-n}{2\tilde{c}(x',0)} \sum_{i,j=1}^{n-1}\hat{g}^{ij}\p_{x_i}\tilde{c}(x',0)\p_{x_j}\tilde{v}^l = 0 & \text{in}\,\, \Om \\
				\tilde{v}^l=f_l, & \text{on}\,\, \partial\Om, \phantom{\Big|}
			\end{array} \right.
		\end{equation}
		where $\tilde{v}^l=\p_{\e_l}\tilde{u}|_{\e=0}$.
		Similarly this can be written as an advection diffusion equation for the real vector field $\tilde{X}$ such that $\tilde{X}h = \frac{1-n}{2}\hat{g}^{ij}\left(\frac{\p_{x_i}c(x',0)}{c(x',0)} + \frac{\p_{x_i}\tilde{c}(x',0)}{\tilde{c}(x',0)}\right)\p_{x_j}h$. Since we know that $\Lambda_g(f)=\Lambda_{\tilde{c}g}(f)$ for all $f\in U_{\delta}$, where $\delta>0$ is sufficiently small, we can apply $\p_{\e_l}|_{\e=0}$ to this, which implies
		\begin{equation*}
			(D\Lambda_g)_0=(D\Lambda_{cg})_0.
		\end{equation*}
		When $n=3$ we use lemma \ref{lemma_first_lin_identification} together with a boundary determination \cite[Proposition 4.1.]{Guillarmou2011} and when $n>3$ we use \cite[Theorem 1.4]{Krupchyk2018} (we assume that the metric $\hat{g}$ is admissible) to deduce that $X=\tilde{X}$ and hence for all $i=1,\ldots,n-1$
		\begin{equation*}
			\frac{\p_{x_i}\tilde{c}(x',0)}{\tilde{c}(x',0)} = 0,\ \text{for}\ x'\in\Om.
		\end{equation*}
		This then implies two things. Firstly, $\tilde{c}(x',0) = \lambda\in (0,\infty)$. Secondly, since solutions to an advection diffusion equation are unique, we have that $v^l=\tilde{v}^l$.
		
		Now the second order linearization for \eqref{BVP_main} is
		\begin{equation}\label{BVP_second_lin}
			\left\{\begin{array}{ll}
				- \Delta_{\hat{g}}w^{kl} + \frac{1-n}{2c(x',0)} \sum_{i,j=1}^{n-1}\hat{g}^{ij}\p_{x_i}c(x',0)\p_{x_j}w^{kl} +  \frac{n-1}{2c(x',0)}\p_{x_n}^3c(x',0)v^kv^l = 0 & \text{in}\,\, \Om \\
				w^{kl}=0, & \text{on}\,\, \partial\Om. \phantom{\Big|}
			\end{array} \right.
		\end{equation}
		For the metric $\tilde{c}g$ the second linearization is
		\begin{equation}\label{BVP_second_lin_tilde}
			\left\{\begin{array}{ll}
				- \Delta_{\hat{g}}\tilde{w}^{kl} + \frac{1-n}{2c(x',0)} \sum_{i,j=1}^{n-1}\hat{g}^{ij}\p_{x_i}c(x',0)\p_{x_j}\tilde{w}^{kl} + \frac{n-1}{2c(x',0)}\p_{x_n}^3c(x',0)v^kv^l \phantom{\bigg|}\\
				+ \frac{n-1}{2\tilde{c}(x',0)}\p_{x_n}^3\tilde{c}(x',0)v^kv^l = 0 & \text{in}\,\, \Om \\
				\tilde{w}^{kl}=0, & \text{on}\,\, \partial\Om. \phantom{\Big|}
			\end{array} \right.
		\end{equation}
		Here $w^{kl}=\p_{\e_k\e_l}^2u|_{\e=0}$ and $\tilde{w}^{kl}\p_{\e_k\e_l}^2\tilde{u}|_{\e=0}$.
		Fix $x_0'\in\Om$. Let now $v^{(0)}$ be a solution to
		\begin{equation}\label{v_zero}
				\Delta_{\hat{g}}v^{(0)} + Xv^{(0)} + qv^{(0)} = 0  \text{ in}\,\, \Om, 
		\end{equation}
		where $q := \frac{1-n}{2c(x',0)}\left(-\frac{1}{c(x',0)}|\nabla_{\hat{g}}c(x',0)|^2_{\hat{g}} + \p_{x_j}\hat{g}^{ij}\p_{x_i}c(x',0) + \hat{g}^{ij}\p_{x_i}c(x',0)\chr{k}{jk}\right)$, so that $v^{(0)}(x_0') \neq 0$. The existence of such a solution can be shown using Runge approximation (see for example \cite[Proposition A.2]{Lassas2020}). More precisely, there exists a solution $\tilde{v}^{(0)}$ of \eqref{v_zero} in a small neighborhood $U\subset\Om$ of $x_0'$ such that $\tilde{v}^{(0)}(x_0') \neq 0$ \cite[Theorem II5.4.1]{Bers1964}. By Runge approximation there exists a solution $v^{(0)}$ with the desired properties.
		Now subtracting equations \eqref{BVP_second_lin}, \eqref{BVP_second_lin_tilde} and integrating against $v^{(0)}$ gives
		\begin{align}\label{int_id1}
			0 &= \int_{\Om} \left(- \Delta_{\hat{g}}(w^{kl} - \tilde{w}^{kl}) + \frac{1-n}{2c(x',0)}\sum_{i,j=1}^{n-1}\hat{g}^{ij}\p_{x_i}c(x',0)\p_{x_j}(w^{kl} - \tilde{w}^{kl})\right)v^{(0)} \\\notag
			&+ \frac{n-1}{2\tilde{c}(x',0)}\p_{x_n}^3\tilde{c}(x',0)v^kv^lv^{(0)}\,dV_{\hat{g}} \\\notag
			&= \int_{\Om} - (w^{kl} - \tilde{w}^{kl})\Delta_{\hat{g}}v^{(0)} + \frac{1-n}{2c(x',0)}\sum_{i,j=1}^{n-1}\hat{g}^{ij}\p_{x_i}c(x',0)\p_{x_j}(w^{kl} - \tilde{w}^{kl})v^{(0)} \\\notag
			&+ \frac{n-1}{2\tilde{c}(x',0)}\p_{x_n}^3\tilde{c}(x',0)v^kv^lv^{(0)} \,dV_{\hat{g}} + \int_{\p\Om}(w^{kl} - \tilde{w}^{kl})\p_{\nu}v^{(0)} - v^{(0)}\p_{\nu}(w^{kl} - \tilde{w}^{kl}) \,dS.
		\end{align}
		Since the volume form is $dV_{\hat{g}}=|\hat{g}|^{\half}\,dx$, where $|\hat{g}|$ is the determinant of $\hat{g}$, we use integration by parts for the second term in the last equality to have
		\begin{align*}
			&+ \int_{\Om} \frac{1-n}{2c(x',0)}\hat{g}^{ij}\p_{x_i}c(x',0)\p_{x_j}(w^{kl} - \tilde{w}^{kl})v^{(0)}|\hat{g}|^{\half}\,dx \\
			&= - \int_{\Om}\bigg(\frac{1-n}{2c(x',0)}\hat{g}^{ij}\p_{x_i}c(x',0)\p_{x_j}v^{(0)}|\hat{g}|^{\half} \\
			&+ \p_{x_j}\left(\frac{1-n}{2c(x',0)}\hat{g}^{ij}\p_{x_i}c(x',0)|\hat{g}|^{\half}\right)v^{(0)}\bigg)(w^{kl} - \tilde{w}^{kl}) \,dx \\
			&+ \int_{\p\Om}(w^{kl} - \tilde{w}^{kl})\frac{1-n}{2c(x',0)}\hat{g}^{ij}\p_{x_i}c(x',0)|\hat{g}|^{\half}v^{(0)}\nu_j\,dS \\
			&= - \int_{\Om}(w^{kl} - \tilde{w}^{kl})(Xv^{(0)} + qv^{(0)})\,dV_{\hat{g}}.
		\end{align*}
		Here we used that $w^{kl} = \tilde{w}^{kl}=0$ on $\p\Om$. Combining this with \eqref{int_id1} gives
		\begin{align*}
			&\int_{\Om} \frac{n-1}{2\tilde{c}(x',0)}\p_{x_n}^3\tilde{c}(x',0)v^kv^lv^{(0)}\,dV_{\hat{g}} \\
			&= \int_{\Om} (w^{kl} - \tilde{w}^{kl})(\Delta_{\hat{g}}v^{(0)} + Xv^{(0)} + qv^{(0)})\,dV_{\hat{g}} \\
			&- \int_{\p\Om}(w^{kl} - \tilde{w}^{kl})\p_{\nu}v^{(0)} - v^{(0)}\p_{\nu}(w^{kl} - \tilde{w}^{kl}) \,dS \\
			&=0.
		\end{align*}
		In the last equality we used $w^{kl} = \tilde{w}^{kl}=0$ on $\p\Om$ and that applying $\p_{\e_k\e_l}^2|_{\e=0}$ to $\p_{\nu}u|_{\p\Om} = \p_{\nu}\tilde{u}|_{\p\Om}$ implies
		\begin{equation*}
			\p_{\nu}w^{kl}|_{\p\Om} = \p_{\nu}\tilde{w}^{kl}|_{\p\Om}.
		\end{equation*}
		Then by Lemma \ref{lemma_density} we have
		\begin{equation*}
			\frac{n-1}{2\tilde{c}(x',0)}\p_{x_n}^3\tilde{c}(x',0)v^{(0)} = 0 \text{ for } x'\in\Om.
		\end{equation*}
		In particular for $x'=x_0'$ this implies $\p_{x_n}^3\tilde{c}(x_0',0) = 0$ but since $x_0'$ was arbitrary we get
		\begin{equation*}
			\p_{x_n}^3\tilde{c}(x',0) = 0 \text{ for } x'\in\Om.
		\end{equation*}
		Now the boundary value problems \eqref{BVP_second_lin} and \eqref{BVP_second_lin_tilde} are the same and thus by uniqueness of solutions $w^{kl} = \tilde{w}^{kl}$ in $\Om$.
		
		The rest of the proof is done very similarly as in \cite{Nurminen2022} but we record the proof here for completeness.
		Next we use induction to show $\partial_{x_n}^k\tilde{c}(x',0)=0$ for all $k\geq3$. By the above this already holds for $k=3.$ Our assumption now is
		\begin{equation*}\label{induction}
			\partial_{x_n}^k\tilde{c}(x',0)=0,\quad x'\in\Om,\va \text{for all}\quad k=3,\ldots,N\in\N, N>3.
		\end{equation*}
		Let us do a subinduction to prove
		\begin{equation*}
			\partial_{l_1\dots l_k}^ku(x',0)=\partial_{l_1\dots l_k}^k\tilde{u}(x',0),\quad x'\in\,\,\Om,
		\end{equation*}
		for all $k=1,\ldots,N$, where $\partial_{l_1\dots l_k}^k=\frac{\partial^k}{\partial_{\e_{l_1}}\dots\partial_{\e_{l_k}}}$. Above we have shown this for $k=1,2$. Assume that it holds for $k\leq K<N$. Then the linearization of order $K+1$ for the metric $g$ is, when evaluated at $\e_1=\dots=\e_{K+1}=0$,
		\begin{align}\label{subinduction}
			&-\Delta_{\hat{g}}\partial_{l_1\dots l_{K+1}}^{K+1}u(x',0) - X\partial_{l_1\dots l_{K+1}}^{K+1}u(x',0) + R_K(u,g(x',0),0)\\\notag
			&+ \frac{n-1}{2c(x',0)}\partial_{x_n}^{K+2}c(x',0)\left(\prod_{k=1}^{K+1}v^{(l_k)}\right) = 0,
		\end{align}
		$x'\in\,\,\Om$. Here $R$ is a polynomial of components of $g$, the derivatives of $g$ and $\partial_{l_1\dots l_k}^ku_1(x',0)$. Also the linearization of order $K+1$ for the metric $\tilde{c}g$ is
		\begin{align}\label{subinduction_tilde}
			&-\Delta_{\hat{g}}\partial_{l_1\dots l_{K+1}}^{K+1}\tilde{u}(x',0) - X\partial_{l_1\dots l_{K+1}}^{K+1}\tilde{u}(x',0) + R_K(\tilde{u},g(x',0),0)\\\notag
			&+ \left(\frac{n-1}{2c(x',0)}\partial_{x_n}^{K+2}c(x',0) + \frac{n-1}{2\tilde{c}(x',0)}\partial_{x_n}^{K+2}\tilde{c}(x',0)\right)\left(\prod_{k=1}^{K+1}v^{(l_k)}\right) = 0.
		\end{align}
		Here $R_K$ could also have terms with $\partial_{x_n}^k\tilde{c}(x',0)$ and the components of $\nabla_{x'}\left(\partial_{x_n}^k\tilde{c}(x',0)\right)$ but terms containing these are zero by the induction assumption.
		Now an integration by parts argument similar to the case of the second linearization and together with Lemma \ref{lemma_density} (choosing $v^3=\ldots=v^{K+1}=1$) gives $\partial_{x_n}^{K+2}\tilde{c}(x',0)=0$.
		
		Subtracting the equations \eqref{subinduction} and \eqref{subinduction_tilde} we get
		\begin{equation*}
			\left\{\begin{array}{ll}
				-\Delta_{\hat{g}}\left(\partial_{l_1\dots l_{K+1}}^{K+1}u(x',0)-\partial_{l_1\dots l_{K+1}}^{K+1}\tilde{u}(x',0)\right) \\
				- X\left(\partial_{l_1\dots l_{K+1}}^{K+1}u(x',0)-\partial_{l_1\dots l_{K+1}}^{K+1}\tilde{u}(x',0)\right) = 0, & \text{in}\,\, \Om \\
				\partial_{l_1\dots l_{K+1}}^{K+1}u(x',0)-\partial_{l_1\dots l_{K+1}}^{K+1}\tilde{u}(x',0)=0, & \text{on}\,\, \partial\Om. \phantom{\Big|}
			\end{array} \right.
		\end{equation*}
		This is true, since by induction assumptions for all $x'\in\Om$ the other terms agree for $k\leq K$. Again, by the uniqueness of solutions, $\partial_{l_1\dots l_{K+1}}^{K+1}u(x',0)=\partial_{l_1\dots l_{K+1}}^{K+1}\tilde{u}(x',0)$, $x'\in\Om$, which ends the subinduction.
		
		Returning to the original induction, the linearization of order $N+1$ at $\e_1=\dots=\e_{N+1}=0$ for the metric $g$ is
		\begin{align*}
			&-\Delta_{\hat{g}}\partial_{l_1\dots l_{N+1}}^{N+1}u(x',0) - X\partial_{l_1\dots l_{N+1}}^{N+1}u(x',0) + R_N(u,g(x',0),0)\\\notag
			&+ \frac{n-1}{2c(x',0)}\partial_{x_n}^{N+2}c(x',0)\left(\prod_{k=1}^{N+1}v^{(l_k)}\right) = 0,
		\end{align*}
		$x'\in\,\,\Om$
		and for the metric $\tilde{c}g$ we have
		\begin{align}\label{subinduction_tilde}
			&-\Delta_{\hat{g}}\partial_{l_1\dots l_{N+1}}^{N+1}\tilde{u}(x',0) - X\partial_{l_1\dots l_{N+1}}^{N+1}\tilde{u}(x',0) + R_N(\tilde{u},g(x',0),0)\\\notag
			&+ \left(\frac{n-1}{2c(x',0)}\partial_{x_n}^{N+2}c(x',0) + \frac{n-1}{2\tilde{c}(x',0)}\partial_{x_n}^{N+2}\tilde{c}(x',0)\right)\left(\prod_{k=1}^{N+1}v^{(l_k)}\right) = 0.
		\end{align}
		Now the subinduction implies that $R_N(u,g(x',0),0)=R_N(\tilde{u},g(x',0),0)$. Thus by subtracting, integrating against $v^{(0)}$ (the solution of \eqref{v_zero}), using integration by parts and that $\partial_{\nu}\partial_{l_1\dots l_{N+1}}^{N+1}u(x',0)|_{\partial\Om}=\partial_{\nu}\partial_{l_1\dots l_{N+1}}^{N+1}\tilde{u}(x',0)|_{\partial\Om}$ we get
		\begin{equation*}
			\int_{\Om}\tilde{c}(x',0)^{-1}\partial_{x_n}^{N+2}\tilde{c}(x',0)v^{(0)}\prod_{k=1}^{N+1}v^{l_k}\,dx'=0.
		\end{equation*}
		Choosing all but two of the functions $v^{l_k}$ to be equal to $1$, then by the completeness of such solutions (Lemma \ref{lemma_density}) implies that
		\begin{equation*}
			\partial_{x_n}^{N+2}\tilde{c}(x_0',0)=0.
		\end{equation*}
		Since $x_0'$ was arbitrary, we have $\partial_{x_n}^{N+2}\tilde{c}(x',0)=0$ for all $x'\in\Om$.
	\end{proof}

	We will next prove Theorem \ref{main thm_partial}. The proof will be very similar to the one above and we will only point out the differences.
	
	\begin{proof}[Proof of Theorem \ref{main thm_partial}]
		Let $\e$ be, as in the previous proof, small enough so that \eqref{BVP_main} has a unique small solution for $f_{\e}:=\sum_{k=1}^{N}f_k$, where $f_k\in C^s(\p\Om)$ with $\spt(f)\subset \Gamma$ and $\norm{f_k}_{C^s(\p\Om)}< \delta$, $\delta>0$.
		
		The first linearizations are \eqref{BVP_first_lin}, \eqref{BVP_first_lin_tilde} and for the corresponding partial DN maps we have
		\begin{equation}\label{DN_map_equiv_first_partial}
			(D\Lambda_g^{\Gamma})_0=(D\Lambda_{cg}^{\Gamma})_0.
		\end{equation}
		Now using Lemma \ref{lemma_first_lin_identification} together with a boundary determination \cite[Proposition $4.1$]{Tzou2017} we get from the first linearization that for all $i=1,\ldots,n-1$
		\begin{equation*}
			\frac{\p_{x_i}\tilde{c}(x',0)}{\tilde{c}(x',0)} = 0,\ \text{for}\ x'\in\Om.
		\end{equation*}
		This then implies two things. Firstly, $\tilde{c}(x',0) = \lambda\in (0,\infty)$. Secondly, since solutions to an advection diffusion equation are unique, we have that $v^l=\tilde{v}^l$.
		
		Let us move to the second order linearization and let $x_0'\in\Om$. We will subtract equations \eqref{BVP_second_lin}, \eqref{BVP_second_lin_tilde} and integrate against the function $v^{(0)}$ that solves
		\begin{equation}\label{special_solution_partial}
			\left\{\begin{array}{ll}
				\Delta_{\hat{g}}v^{(0)} - Xv^{(0)} - qv^{(0)} = 0 & \text{in}\,\, \Om \\
				v^{(0)}=g, & \text{on}\,\, \Gamma   \\
				v^{(0)}=0, & \text{on}\,\, \partial\Om\setminus\Gamma. \phantom{\Big|}
			\end{array} \right.
		\end{equation}
		for $g\in C^{\infty}_c(\Gamma)$, $g\geq0$, $g\not\equiv0$ and that $v^{(0)}(x_0')\neq0$. This function can be constructed exactly as before. We now do the integration, where the difference will be on how to deal with the boundary terms:
		\begin{align*}
			&\int_{\Om} \frac{n-1}{2\tilde{c}(x',0)}\p_{x_n}^3\tilde{c}(x',0)v^kv^lv^{(0)}\,dV_{\hat{g}} \\
			&= \int_{\Om} -(w^{kl} - \tilde{w}^{kl})(-\Delta_{\hat{g}}v^{(0)} + Xv^{(0)} + qv^{(0)})\,dV_{\hat{g}} \\
			&- \int_{\p\Om}(w^{kl} - \tilde{w}^{kl})\p_{\nu}v^{(0)} - v^{(0)}\p_{\nu}(w^{kl} - \tilde{w}^{kl}) \,dS \\
			&+ \int_{\p\Om}(w^{kl} - \tilde{w}^{kl})\frac{1-n}{2c(x',0)}\hat{g}^{ij}\p_{x_i}c(x',0)|\hat{g}|^{\half}v^{(0)}\nu_j\,dS. \\
		\end{align*}
		Since $v^{(0)}$ solves \eqref{special_solution_partial} and $(w^{kl} - \tilde{w}^{kl})|_{\p\Om}=0$ the first and last integral vanish and also the first term in the second integral vanishes. For the remaining part we divide the integral in two parts
		\begin{equation*}
			\int_{\p\Om}v^{(0)}\p_{\nu}(w^{kl} - \tilde{w}^{kl}) \,dS = \int_{\Gamma}v^{(0)}\p_{\nu}(w^{kl} - \tilde{w}^{kl}) \,dS + \int_{\p\Om\setminus\Gamma}v^{(0)}\p_{\nu}(w^{kl} - \tilde{w}^{kl}) \,dS.
		\end{equation*}
		This is zero since $v^{(0)}|_{\p\Om\setminus\Gamma}=0$ and from \eqref{DN_map_equiv_first_partial} we have that $\p_{\nu}(w^{kl} - \tilde{w}^{kl})|_{\Gamma}=0$. Hence
		\begin{equation*}
			\int_{\Om} \frac{n-1}{2\tilde{c}(x',0)}\p_{x_n}^3\tilde{c}(x',0)v^kv^lv^{(0)}\,dV_{\hat{g}}=0
		\end{equation*}
		for all $v_k$, $v^l$ solving \eqref{BVP_first_lin}. Continuing as in the proof of Theorem \ref{main thm} we get
		\begin{equation*}
			\p_{x_n}^3\tilde{c}(x',0) = 0 \text{ for } x'\in\Om.
		\end{equation*}
		and this then implies that $w^{kl} = \tilde{w}^{kl}$ in $\Om$.
		
		Proceeding with induction as before and taking further care about the boundary terms in the integrals will conclude the proof.
	\end{proof}

	\printbibliography
\end{document}